\documentclass[11pt,leqno]{article}

\usepackage[active]{srcltx}

\usepackage{amsfonts,latexsym,amsmath,amssymb,amsthm}

\usepackage{fullpage}
\usepackage{dsfont}
\newtheorem{theorem}{Theorem}[section]
\newtheorem{lemma}[theorem]{Lemma}

\newtheorem{corollary}[theorem]{Corollary}
\newtheorem{remark}[theorem]{Remark}

\theoremstyle{definition}

\theoremstyle{remark}

\newtheorem*{note*}{Note}

\numberwithin{equation}{section}

\makeatletter
\newcommand{\rank}{\mathop{\operator@font rank}}
\newcommand{\conv}{\mathop{\operator@font conv}}
\newcommand{\vol}{\mathop{\operator@font vol}}
\newcommand{\onetagright}{\tagsleft@false}
\makeatother

\newcommand{\ls}{\leqslant}
\newcommand{\gr}{\geqslant}
\renewcommand{\epsilon}{\varepsilon}

\newcommand{\pa}{\hskip 0.5truecm}

\newcommand{\R}{\mathbb{R}}

\usepackage{color}

\usepackage{hyperref}

\usepackage{setspace}
\begin{document}
\small

\title{\bf Inequalities for sections and projections\\ of log-concave functions}

\medskip

\author{Natalia Tziotziou}

\date{}
\maketitle

\begin{abstract}\footnotesize We extend several geometric inequalities for sections and projections of convex bodies to the setting of integrable log-concave functions. In particular, we introduce suitable generalizations of the affine and dual affine quermassintegrals of a log-concave function $f$ and obtain upper and lower bounds for them in terms of the integral $\|f\|_1$. We also establish estimates for sections and projections of log-concave functions in the spirit of the lower dimensional Busemann-Petty and Shephard problems
and extend to the log-concave setting the affirmative answer to a variant of these problems, proposed by V.~Milman. The main objective of this work is to show that the assumption of
log-concavity leads to inequalities whose constants are of the same order as those in the corresponding geometric inequalities.
\end{abstract}

\section{Introduction}

\pa In this article, we present extensions of some geometric inequalities about sections and projections
of convex bodies to the setting of integrable log-concave functions. Various extensions of this type exist 
for even more general classes of functions. Our aim is to show that the assumption of log-concavity leads to much
better constants, namely of the same order as those in the corresponding geometric inequalities.
We say that a function $f:{\mathbb R}^n\to [0,+\infty)$
is log-concave if $f=e^{-\varphi}$ where $\varphi:{\mathbb R}^n\to (-\infty ,\infty]$ is convex and
lower semicontinuous. We also say that $f$ is a geometric log-concave function if it is log-concave and
$\|f\|_{\infty}=f(0)=1$. We denote by ${\mathcal F}({\mathbb R}^n)$ the class of log-concave integrable functions
on ${\mathbb R}^n$ and by ${\mathcal F}_0({\mathbb R}^n)$ the class of geometric log-concave integrable functions.
Let $f\in {\mathcal F}({\mathbb R}^n)$. Given $E\in G_{n,k}$, where $G_{n,k}$ is the Grassmann manifold of $k$-dimensional subspaces of ${\mathbb R}^n$, the ``section" of $f$ with $E$ is the restriction $f\big|_E$ of $f$ onto $E$ and the ``projection" or ``shadow"
of $f$ onto $E$ is the function
$$P_Ef(x):=\max\{f(y):y\in x+E^{\perp}\},\quad x\in E$$
where $E^{\perp}$ is the orthogonal subspace of $E$.

We start with a very brief description of our results. First we consider suitable generalizations of the affine and dual
affine quermassintegrals of a log-concave function $f$ and provide upper and lower estimates for them in terms of the integral
$$\|f\|_1=\int_{{\mathbb R}^n}f(x)\,dx$$
of $f$. Then, we extend to log-concave functions the affirmative answer to a variant of the Busemann-Petty and
Shephard problems, proposed by V.~Milman. We also provide estimates for sections of log-concave functions
in the spirit of the lower dimensional Busemann-Petty problem. Similarly, we give estimates for projections of log-concave functions
in the spirit of the lower dimensional Shephard problem. More information about each of these questions,
and about the classical geometric counterparts of our results, is provided separately in the corresponding
sections of the article.

In Section~\ref{section-3} we consider the affine and dual affine quermassintegrals of log-concave integrable functions.
Recall that for every convex body $K$ or, more generally, for any bounded Borel set
in $\mathbb{R}^n$, and for any $1\ls k \ls n-1$, the $k$-th dual affine quermassintegral of $K$ is defined by
$$\Psi_k(K):= \left(\int_{G_{n,k}} |K\cap E^{\perp}|^n \,d\nu_{n,k}(E) \right)^{\frac{1}{kn}},$$
where $|\cdot|$ denotes volume in the appropriate dimension and $\nu_{n,k}$ is the Haar probability measure
on $G_{n,k}$. We consider the following natural generalization of $\Psi_k$ for a log-concave integrable function:
$$\Psi_k(f):=\left(\int_{G_{n,k}} \|f\,\big|_{E^{\perp}}\|_1^n \,d\nu_{n,k}(E) \right)^{\frac{1}{kn}}.$$
Our estimates for $\Psi_k(f)$ are the following.

\begin{theorem}\label{th:intro-dual-functional}Let $f:{\mathbb R}^n\to [0,+\infty)$ be a geometric log-concave integrable
function. Then,
$$c\|f\|_1^{\frac{n-k}{kn}}\ls\Psi_k(f)\ls \sqrt{e}\|f\|_1^{\frac{n-k}{kn}}$$
for every $1\ls k\ls n-1$, where $c>0$ is an absolute constant.
\end{theorem}

Next, recall that for every convex body $K$ or, more generally, for any bounded Borel set in $\mathbb{R}^n$,
and for any $1\ls k \ls n-1$, the $k$-th affine quermassintegral of $K$ is defined by
$$\Phi_k(K) := \left( \int_{G_{n,k}} |P_E(K)|^{-n} \,d\nu_{n,k}(E) \right)^{-\frac{1}{kn}}.$$
We generalize the definition of $\Phi_k$ as follows: for any log-concave integrable function we set
\begin{equation}\label{eq:def-phi}\Phi_k(f)=\left(\int_0^{\infty}(\Phi_k(R_t(f)))^n\,dt\right)^{1/n},\end{equation}
where $R_t(f)=\{x:f(x)\gr t\}$, $t>0$. We note at this point that if $f=\mathds{1}_K$ is the indicator function
of a convex body $K$ in ${\mathbb R}^n$ then $\Phi_k(\mathds{1}_K)=\Phi_k(K)$ for all $1\ls k\ls n-1$ (as we explain in Section~\ref{section-3}, this follows from the observation that $P_E(R_t(f))=R_t(P_E\mathds{1}_K)=P_E(K)$ for
all $0\ls t<1$ and  $R_t(P_E\mathds{1}_K)=\varnothing$ for all $t\gr 1$). Therefore, our definition of $\Phi_k(f)$
in \eqref{eq:def-phi} generalizes the definition of the $k$-th affine quermassintegral of a convex body.
We obtain the following estimates.

\begin{theorem}\label{th:intro-functional}Let $f:{\mathbb R}^n\to [0,+\infty)$ be a geometric log-concave integrable
function. Then,
$$c_1\sqrt{n/k}\,\|f\|_1^{\frac{1}{n}}\ls \Phi_k(f)\ls c_2\sqrt{n/k}\phi_{n,k}\,\|f\|_1^{\frac{1}{n}}$$
for every $1\ls k\ls n-1$, where $\phi_{n,k}=\min\left\{\log n, \frac{n}{k}\sqrt{\log(en/k)}\right\}$ and $c_1,c_2>0$ are absolute constants.
\end{theorem}

In Section~\ref{section-4} we study a variant of the Busemann-Petty and Shephard problems, proposed by V.~Milman
in the setting of convex bodies: If $K$ and $T$ are origin-symmetric convex bodies
in ${\mathbb R}^n$ that satisfy $|P_{\xi^{\perp }}(K)|\ls |T\cap\xi^{\perp }|$
for all $\xi\in S^{n-1}$, does it follow that $|K|\ls |T|$?
Giannopoulos and Koldobsky proved in \cite{Giannopoulos-Koldobsky-2017} that the answer to this question
(and in fact to the lower dimensional analogue of the question) is affirmative. Moreover, one can relax the symmetry assumption
and even the assumption of convexity for $T$: If $K$ is a convex body in ${\mathbb R}^n$ and $T$ is a bounded
Borel subset of ${\mathbb R}^n$ such that $|P_E(K)|\ls |T\cap E|$ for some $1\ls k\ls n-1$ and for all
$E\in G_{n,n-k}$, then $|K|\ls |T|$.

We extend this result to log-concave integrable functions.

\begin{theorem}\label{th:vitali}Let $f,g:{\mathbb R}^n\to [0,+\infty)$ be geometric log-concave integrable functions
such that, for some $1\ls k\ls n-1$, we have that
$$\|P_Ef\|_1\ls \|g\big|_E\|_1\qquad\hbox{for all}\;\;E\in G_{n,n-k}.$$
Then,
\begin{equation}\label{eq:vitali-main}\|f\|_1\ls \frac{n!}{[(n-k)!]^{\frac{n}{n-k}}}\,\|g\|_1.\end{equation}
\end{theorem}

Using Stirling's formula we check that the constant $n!/[(n-k)!]^{\frac{n}{n-k}}$ that appears in
\eqref{eq:vitali-main} is of the order of $C^{\frac{kn}{n-k}}$.

\smallskip

In Section~\ref{section-5} we extend to log-concave integrable functions a number of inequalities related to the Busemann-Petty problem and the slicing problem. They all follow from the next general inequality about the Radon transform on convex sets.

\begin{theorem}\label{th:radon-quotient}Let $f,g:{\mathbb R}^n\to [0,\infty)$ be non-negative integrable functions such that
$f$ is log-concave with $f(0)>0$ and $\|g\|_\infty=g(0)=1$.
If $K$ is a convex body in ${\mathbb R}^n$ and $T$ is a compact subset of ${\mathbb R}^n$ with $0\in K\cap T$, then
\begin{equation}\label{eq:radon-quotient}
\left(\frac{f(0)^{-1}\int_Kf(x)\,dx}{\int_T g(x)\,dx}\right)^{\frac{n-k}{n}}  \ls C^k
\max_{E\in G_{n,n-k}} \frac{f(0)^{-1}\int_{K\cap E} f(x)\,dx}{\int_{T\cap E} g(x)\,dx},
\end{equation}
where $C>0$ is an absolute constant.
\end{theorem}

We discuss several consequences of Theorem~\ref{th:radon-quotient} in Section~\ref{section-5}. In particular,
we obtain a lower estimate for the sup-norm of the Radon transform which sharpens Koldobsky's slicing inequality
for arbitrary functions (see~\cite{Koldobsky-AIM-2015}) under the log-concavity assumption.

\begin{theorem}\label{th:intro-slicing}Let $f\in {\mathcal F}_0({\mathbb R}^n)$. Then, for every $1\ls k\ls n-1$ we have that
$$\int_Kf(x)\,dx\ls C^k |K|^{\frac kn}\max_{E\in G_{n,n-k}} \int_{K\cap E} f(x)\,dx,$$
where $C>0$ is an absolute constant.
\end{theorem}

In Section~\ref{section-6} we obtain a Shephard-type estimate for log-concave integrable functions.
Let $1\ls k\ls n-1$ and let $S_{n,k}$ be the smallest constant $S >0$ with the following property: For
every pair of convex bodies $K$ and $T$ in ${\mathbb R}^n$ that satisfy
$|P_E(K)|\ls |P_E(T)|$ for all $E\in G_{n,n-k}$, one has that $|K|^{\frac{n-k}{n}}\ls S^k\,|T|^{\frac{n-k}{n}}$.
The final (negative) answer to the classical Shephard's problem is the fact that $S_{n,1}\approx\sqrt{n}$ (see
Section~\ref{section-6} for more details and references). General estimates for $S_{n,k}$ were
obtained in \cite{Giannopoulos-Koldobsky-2017} where it was shown that if
$$\tilde{S}_{n,k}=\min\left\{\sqrt{\tfrac{n}{n-k}}\ln\left (\tfrac{en}{n-k}\right),\ln n\right\}$$
then $S_{n,k}\ls (c_1\tilde{S}_{n,k})^{\frac{n-k}{k}}$. In particular, $S_{n,k}\ls C^{\frac{n-k}{k}}$
if $\frac{k}{n-k}$ is bounded. We provide a third estimate, namely that $S_{n,k}\ls\sqrt{n}$ for all $n$
and $k$, which is consistent with the estimate for $S_{n,1}$ and better than the previous ones if $k\ll n/\log n$.
Our main goal is to obtain an analogue of the above estimates for $S_{n,k}$ for geometric log-concave
integrable functions.

\begin{theorem}\label{th:intro-shephard}Let $f,g\in {\mathcal F}_0({\mathbb R}^n)$ and $1\ls k\ls n-1$. Assume that $|R_t(P_Ef)|\ls |R_t(P_Eg)|$ for all $E\in G_{n,n-k}$ and $0\ls t<1$. Then,
$$\|f\|_1^{\frac{n-k}{n}}\ls S_{n,k}^k\|g\|_1^{\frac{n-k}{n}}.$$
\end{theorem}

Theorem~\ref{th:intro-shephard} extends the known Shephard-type estimates for convex bodies. To see this, we note
that if $f=\mathds{1}_K$ and $g=\mathds{1}_T$ are the indicator functions of two convex bodies $K$ and $T$ in ${\mathbb R}^n$ then
the assumption in Theorem~\ref{th:intro-shephard} is equivalent to the assumption $|P_E(K)|\ls |P_E(T)|$
in the original Shephard's problem.

\smallskip

We close this introductory section with a few comments on some of the main tools that are used for the
proofs of the above results. A very fruitful idea in order to study the geometric properties of a log-concave function $f$
on ${\mathbb R}^n$ with $f(0)>0$ is to use the family of bodies $K_p(f)$, introduced by K.~Ball in \cite{Ball-1988}. In Section~\ref{section-2} we recall their definition: $K_p(f)$ is the star body with radial function
$$\rho_{K_p(f)}(x)=\left (\frac{1}{f(0)}\int_0^{\infty}pr^{p-1}f(rx)\,dr\right )^{1/p}$$ for $x\neq 0$. The
log-concavity of $f$ implies that the bodies $K_p(f)$ are convex. Direct computation
shows that
$$|K_n(f)|=\frac{1}{f(0)}\int_{{\mathbb R}^n}f(x)\,dx$$
and
$$|K_{n-k}(f)\cap E|=\frac{1}{f(0)}\int_Ef(x)\,dx$$
for any $1\ls k\ls n-1$ and $E\in G_{n,n-k}$. Moreover (see Lemma~\ref{lem:any-g} and
Lemma~\ref{lem:log-concave-f}), one can compare the volumes of $|K_{n-k}(f)|$ and $|K_n(f)|$. These facts allow us
to translate computations involving sections of $f$ to computations about the sections of the convex
body $K_{n-k}(f)$ and exploit what is known in the setting of convex bodies. This idea is used in the
proof of Theorem~\ref{th:intro-dual-functional} and Theorem~\ref{th:radon-quotient}.

In order to study the projections of a log-concave function $f$ on ${\mathbb R}^n$ we use
mixed integrals of functions. If $f_1,\ldots ,f_n:{\mathbb R}^n\to [0,\infty)$
are bounded integrable functions such that $R_t(f_j)=\{f_j\gr t\}$ is a compact convex set
for all $1\ls j\ls n$ and $t>0$, their mixed integral (see \cite{VMilman-Rotem-2013}) is defined by
$$V(f_1,f_2,\ldots ,f_n):=\int_0^{\infty}V(R_t(f_1),R_t(f_2),\ldots ,R_t(f_n))\,dt,$$
where $V(R_t(f_1),R_t(f_2),\ldots ,R_t(f_n))$ is the classical mixed volume of the sets $R_t(f_j)$.
The $k$-th quermassintegral of $f$ is defined as $W_k(f)=V((f,n-k),(\mathds{1}_{B_2^n},k))$. Direct computation shows that
$$V(f,\ldots ,f)=\|f\|_1$$
and
$$W_k(f) =\frac{\omega_n}{\omega_{n-k}}\int_{G_{n,n-k}}\|P_E(f)\|_1\,d\nu_{n,n-k}(E).$$
Moreover, the quermassintegrals $W_k(f)$ can be compared via Aleksandrov-type inequalities
(see \eqref{eq:aleksandrov-2}). These facts allow us to translate computations involving projections of $f$
to computations about the quermassintegrals of the convex
bodies $R_t(f)$ and exploit what is known in the setting of convex bodies. This idea is used in the
proof of Theorem~\ref{th:intro-functional} and Theorem~\ref{th:vitali}.

\section{Background information and auxiliary results}\label{section-2}

\pa We write $\langle\cdot ,\cdot\rangle $ for the standard inner product in ${\mathbb R}^n$ and $|\cdot|$ for the Euclidean norm.
Lebesgue measure in ${\mathbb R}^n$ is also denoted by $|\cdot |$. We denote the Euclidean unit ball and unit sphere by $B_2^n$ and $S^{n-1}$ respectively, and write $\sigma $ for the rotationally invariant probability measure on $S^{n-1}$. We define $\omega_n=|B_2^n|$.
The Grassmann manifold $G_{n,k}$ of $k$-dimensional subspaces of ${\mathbb R}^n$ is equipped with the Haar probability
measure $\nu_{n,k}$. The letters $c, c^{\prime },c_j,c_j^{\prime }$ etc. denote absolute positive constants whose value may change from line to line.

A convex body in ${\mathbb R}^n$ is a compact convex set $K\subset {\mathbb R}^n$ with non-empty interior. We say that $K$ is centrally symmetric if $-K=K$ and that $K$ is centered if the barycenter ${\rm bar}(K)=\frac{1}{|K|}\int_Kx\,dx$ of $K$ is at the origin.
If $K$ is a convex body with $0\in {\rm int}(K)$ then the radial function $\varrho_K$ of $K$ is defined for all $x\neq 0$ by $\varrho_K(x)=\sup \{t>0:tx\in K\}$ and
the support function of $K$ is defined by $h_K(x) = \sup\{\langle x,y\rangle :y\in K\}$ for all $x\in {\mathbb R}^n$.
The polar body $K^{\circ }$ of a convex body $K$ in ${\mathbb R}^n$ with $0\in {\rm int}(K)$ is the convex body
\begin{equation*}
K^{\circ}:=\bigl\{y\in {\mathbb R}^n: \langle x,y\rangle \ls 1\;\hbox{for all}\; x\in K\bigr\}.
\end{equation*}
A convex body $K$ in ${\mathbb R}^n$ is called isotropic if it has volume $1$, it is centered, and its inertia matrix is a multiple of the
identity matrix. This is equivalent to the fact that there exists a constant $L_K>0$, the isotropic constant of $K$, such that
\begin{equation*}\|\langle \cdot ,\xi\rangle\|_{L_2(K)}^2:=\int_K\langle x,\xi\rangle^2dx =L_K^2\end{equation*}
for all $\xi\in S^{n-1}$.

A Borel measure $\mu$ on $\mathbb R^n$ is called log-concave if $\mu(\lambda
A+(1-\lambda)B) \gr \mu(A)^{\lambda}\mu(B)^{1-\lambda}$ for any pair of compact sets $A$
and $B$ in ${\mathbb R}^n$ and any $\lambda \in (0,1)$. A function
$f:\mathbb R^n \rightarrow [0,\infty)$ is called log-concave if
its support $\{f>0\}$ is a convex set in ${\mathbb R}^n$ and the restriction of $\ln{f}$ to it is concave.
If $f$ has finite positive integral then there exist constants $A,B>0$ such that $f(x)\ls Ae^{-B|x|}$
for all $x\in {\mathbb R}^n$ (see \cite[Lemma~2.2.1]{BGVV-book}). In particular, $f$ has finite moments
of all orders. A result of Borell \cite{Borell-1974} shows that if a probability measure $\mu $ is log-concave and $\mu (H)<1$ for every
hyperplane $H$ in ${\mathbb R}^n$, then $\mu $ has a log-concave density $f_{{\mu }}$.
We say that $\mu $ is even if $\mu (-B)=\mu (B)$ for every Borel subset $B$ of ${\mathbb R}^n$. We also say that $\mu $ is centered,
and write ${\rm bar}(\mu)=0$, if
\begin{equation*}
\int_{\mathbb R^n} \langle x, \xi \rangle d\mu(x) = \int_{\mathbb R^n} \langle x, \xi \rangle f_{\mu}(x) dx = 0
\end{equation*} for all $\xi\in S^{n-1}$. Fradelizi has shown in \cite{Fradelizi-1997} that if $\mu $ is a centered log-concave
probability measure on ${\mathbb R}^n$ then
\begin{equation}\label{eq:frad-2}\|f_{\mu }\|_{\infty }\ls e^nf_{\mu }(0).\end{equation}
Note that if $K$ is a convex body in $\mathbb R^n$ then the Brunn-Minkowski inequality implies that the indicator function
$\mathds{1}_{K} $ of $K$ is the density of a log-concave measure, the Lebesgue measure on $K$.

Let $\mu $ and $\nu$ be two Borel measures on ${\mathbb R}^n$. If $T:{\mathbb
R}^n\to {\mathbb R}^n$ is a measurable function which is defined
$\nu $-almost everywhere and satisfies
\begin{equation*}\mu (B)=\nu (T^{-1}(B))\end{equation*}
for every Borel subset $B$ of ${\mathbb R}^n$ then we say that $T$
pushes forward $\nu $ to $\mu $ and write $T_*\nu=\mu$. It is
easy to see that $T_*\nu =\mu $ if and only if for every bounded Borel
measurable function $g:{\mathbb R}^n\to {\mathbb R}$ we have
\begin{equation*}\int_{{\mathbb R}^n}g(x)d\mu (x)=\int_{{\mathbb R}^n}g(T(y))d\nu (y).\end{equation*}
If $\mu $ is a log-concave measure on ${\mathbb R}^n$ with density $f_{\mu}$, we define the isotropic constant of $\mu $ by
\begin{equation*}
L_{\mu }:=\left (\frac{\sup_{x\in {\mathbb R}^n} f_{\mu} (x)}{\int_{{\mathbb
R}^n}f_{\mu}(x)dx}\right )^{\frac{1}{n}} [\det \textrm{Cov}(\mu)]^{\frac{1}{2n}},\end{equation*}
where $\textrm{Cov}(\mu)$ is the covariance matrix of $\mu$ with entries
\begin{equation*}\textrm{Cov}(\mu )_{ij}:=\frac{\int_{{\mathbb R}^n}x_ix_j f_{\mu}
(x)\,dx}{\int_{{\mathbb R}^n} f_{\mu} (x)\,dx}-\frac{\int_{{\mathbb
R}^n}x_i f_{\mu} (x)\,dx}{\int_{{\mathbb R}^n} f_{\mu}
(x)\,dx}\frac{\int_{{\mathbb R}^n}x_j f_{\mu}
(x)\,dx}{\int_{{\mathbb R}^n} f_{\mu} (x)\,dx}.\end{equation*} We say
that a log-concave probability measure $\mu $ on ${\mathbb R}^n$
is isotropic if it is centered and $\textrm{Cov}(\mu )=I_n$,
where $I_n$ is the identity $n\times n$ matrix. In this case, $L_{\mu }=\|f_{\mu }\|_{\infty }^{1/n}$.
It is known that for every $\mu $ there exists an affine transformation $T$
such that $T_{\ast }\mu $ is isotropic. 

The hyperplane conjecture asks if there exists an absolute constant $C>0$ such that
\begin{equation*}L_n:= \max\{ L_{\mu }:\mu\ \hbox{is an isotropic log-concave probability measure on}\ {\mathbb R}^n\}\ls C\end{equation*}
for all $n\gr 2$. The classical estimates $L_n\ls c\sqrt[4]{n}\ln n$ by Bourgain \cite{Bourgain-1991}
and, fifteen years later, $L_n\ls c\sqrt[4]{n}$ by Klartag \cite{Klartag-2006} remained the best known until
2020. In a breakthrough work, Chen \cite{Chen-2021} 
proved that for any $\epsilon >0$
one has $L_n\ls n^{\epsilon}$ for all large enough $n$. This development was the starting point for
a series of important works, including a recent technical breakthrough by Guan \cite{Guan}, that culminated
in the final affirmative answer to the problem. Very recently, Klartag and Lehec \cite{KL}
showed that $L_n\ls C$. Shortly thereafter, Bizeul~\cite{Bizeul-2025} provided an alternative proof.

As mentioned in the introduction, a main tool in our work is the family of $K_p$-bodies of a function,
introduced by K.~Ball in \cite{Ball-1988}. Given a measurable function $f:\mathbb R^n\to [0,\infty)$
with $f(0) >0$, for any $p>0$ we define the set
\begin{equation}\label{eq:ball-bodies}
K_p(f)=\left\{ x\in \mathbb R^n : \int_0^\infty f(rx)r^{p-1} \, dr
\gr \frac{f(0)}{p}\right\}.
\end{equation}
From the definition it follows that the radial function of $K_p(f)$ is given by
\begin{equation}\label{eq:rho_Kp(f)}
\rho_{K_p(f)}(x)=\left (\frac{1}{f(0)}\int_0^{\infty}pr^{p-1}f(rx)\,dr\right )^{1/p}\end{equation} for $x\neq 0$.
It was proved by K.~Ball in \cite{Ball-1988} that if $f$ is also log-concave and integrable then, for every $p>0$, $K_p(f)$ is a convex set 

Let $0<p<q$. The next two lemmas establish inclusions between $K_p(f)$ and $K_q(f)$.

\begin{lemma}\label{lem:any-g}
Let $f:\mathbb R^n\to [0,\infty)$ be a bounded measurable function such that $f(0)>0$. If $0 < p\ls q$, then
\begin{equation}\label{any-g}K_{p}(f)\subseteq
\left(\frac{\|f\|_{\infty}}{f(0)}\right)^{\frac{1}{p}-\frac{1}{q}}K_{q}(f).
\end{equation}
\end{lemma}

\begin{proof}For any $x\neq 0$ consider the function $f_x:[0,\infty)\to [0,\infty)$ with $f_x(r)=f(rx)$. It is known that
the function
$$G(p):=\left(\frac{1}{\|f_x\|_{\infty}}\int_0^{\infty}pr^{p-1}f_x(r)\,dr\right)^{1/p}$$
is increasing on $(0,\infty)$. A proof may be found in \cite[Lemma~2.2.4]{BGVV-book}; note that the log-concavity
of $f$ is not required for the argument. Applying this fact one can check (see \cite[Proposition~2.5.7]{BGVV-book}) that
$$\rho_{K_q(f)}(x)\gr\left(\frac{\|f_x\|_{\infty}}{f(0)}\right)^{1/q-1/p}\rho_{K_p(f)}(x)$$
for all $0<p\ls q$, and the lemma follows if we also note that
$\|f_x\|_{\infty}\ls\|f\|_{\infty}$.\end{proof}

\begin{lemma}\label{lem:log-concave-f}
Let $f:\mathbb R^n\to [0,\infty)$ be a log-concave function such
that $f(0)>0$. If $0 < p\ls q$, then
\begin{equation}\label{log-concave-f} \frac{\Gamma(p+1)^{\frac{1}{p}}}{\Gamma(q+1)^{\frac{1}{q}}} K_q(f)\subseteq
K_{p}(f).\end{equation}
\end{lemma}

\begin{proof}For any $x\neq 0$ consider the log-concave function $f_x:[0,\infty)\to [0,\infty)$ with $f_x(r)=f(rx)$. It is known that
the function
$$F(p):=\left(\frac{1}{f_x(0)\Gamma(p)}\int_0^{\infty}r^{p-1}f_x(r)\,dr\right)^{1/p}$$
is decreasing on $(0,\infty)$. A proof may be found in \cite[Theorem~2.2.3]{BGVV-book}. Applying this fact one can check (see \cite[Proposition~2.5.7]{BGVV-book}) that
$$\rho_{K_q(f)}(x)\ls\frac{\Gamma(q+1)^{1/q}}{\Gamma(p+1)^{1/p}}\rho_{K_p(f)}(x)$$
for all $0<p\ls q$.\end{proof}

We shall use the fact that
\begin{align}\label{eq:quotient-1}
|K_n(f)| &=\frac{1}{n}\int_{S^{n-1}}\rho_{K_n(f)}(\xi)^nd\xi\\
\nonumber &=\frac{1}{n}\int_{S^{n-1}}\frac{1}{f(0)}\int_0^{\infty}nr^{n-1}f(r\xi)\,dr\,d\xi
=\frac{1}{f(0)}\int_{{\mathbb R}^n}f(x)\,dx.
\end{align}
Similarly, for any $1\ls k\ls n-1$ and $E\in G_{n,n-k}$ we see that
\begin{align}\label{eq:quotient-2}
|K_{n-k}(f)\cap E| &=\frac{1}{n-k}\int_{S^{n-1}\cap E}
\rho_{K_{n-k}(f)}(\xi)^{n-k}d\xi\\
\nonumber &=\frac{1}{n-k}\int_{S^{n-1}\cap E}\frac{1}{f(0)}\int_0^{\infty}(n-k)r^{n-k-1}f(r\xi)\,dr\,d\xi \\
\nonumber &=\frac{1}{f(0)}\int_Ef(x)\,dx.
\end{align}

We will also briefly consider $s$-concave measures. We say that a measure $\mu $ on
${\mathbb R}^n$ is $s$-concave for some $-\infty\ls s\ls 1/n$ if
\begin{equation}\label{eq:s-concave-measure-1}\mu ((1-\lambda )A+\lambda B)\gr
((1-\lambda )\mu^{s}(A)+\lambda \mu^{s}(B))^{1/s}\end{equation}
for any pair of compact sets $A,B$ in ${\mathbb R}^n$ with $\mu (A)\mu (B)>0$ and any $\lambda\in (0,1)$.
We can also consider the limiting cases $s =0$, where the right-hand side in \eqref{eq:s-concave-measure-1}
should be understood as $\mu(A)^{1-\lambda }\mu (B)^{\lambda }$, and $s=-\infty $, where the right-hand side in \eqref{eq:s-concave-measure-1} becomes $\min\{\mu (A),\mu (B)\}$. Note that $0$-concave measures are the log-concave
measures and that if $\mu $ is $s$-concave and $s^{\prime}\ls s$ then $\mu $ is also $s^{\prime}$-concave.

A function $f:\mathbb R^n\to [0,\infty)$ is called $\gamma $-concave for some $\gamma\in [-\infty ,\infty ]$
if
\begin{equation*}f((1-\lambda )x+\lambda y)\gr
((1-\lambda )f^{\gamma }(x)+\lambda f^{\gamma }(y))^{1/\gamma }\end{equation*}
for all $x,y\in {\mathbb R}^n$ with $f(x)f(y)>0$ and all $\lambda\in (0,1)$. One can also
define the cases $\gamma =0,+\infty $ appropriately. Borell \cite{Borell-1975}
showed that if $\mu $ is a measure on ${\mathbb R}^n$ and the affine subspace $F$ spanned by the support
${\rm supp}(\mu )$ of $\mu $ has dimension ${\rm dim}(F)=n$ then for every $-\infty \ls s <1/n$ we have that
$\mu $ is $s$-concave if and only if it has a non-negative density $f\in L_{{\rm loc}}^1({\mathbb R}^n,dx)$
and $f$ is $\gamma $-concave, where $\gamma =\frac{s}{1-sn}\in [-1/n,+\infty )$.

We shall extend some of our results to densities of $s$-concave measures with $s\in (-\infty,0)$. Note that these classes of
functions are strictly larger than the class of log-concave functions. Let $\mu$ be $s$-concave for some $s\in (-\infty,0)$.
Then, the density $f$ of $\mu$ is $-\frac{1}{\alpha}$-concave,
where $\displaystyle -\frac{1}{\alpha}=\frac{s}{1-sn}$, or equivalently,
$$\alpha =n-\frac{1}{s}>n.$$
Assume that $f(0)>0$ and, for any $0<p<\alpha$, define the star body $K_p(f)$ as in \eqref{eq:ball-bodies}. It is proved
in \cite{Fradelizi-Guedon-Pajor-2014} that one has an analogue of Lemma~\ref{lem:log-concave-f}: For any $0<p\ls q<\alpha $,
\begin{equation}\label{eq:s<0}K_q(f)\subseteq \frac{(qB(q,\alpha -q))^{1/q}}{(pB(p,\alpha -p))^{1/p}}K_p(f),\end{equation}
where $\displaystyle B(x,y)=\frac{\Gamma(x)\Gamma(y)}{\Gamma(x+y)}$ is the Beta function. Note that if $\alpha\gr q+1$ then
$$c_1\frac{q}{p}\ls\frac{(qB(q,\alpha -q))^{1/q}}{(pB(p,\alpha -p))^{1/p}}\ls c_2\frac{q}{p}$$
where $c_1,c_2>0$ are absolute constants. This follows from \cite[Lemma~11]{Fradelizi-Guedon-Pajor-2014}.

We refer to the classical monograph of Schneider \cite{Schneider-book} for the theory of convex bodies and to
the books \cite{AGA-book} and \cite{BGVV-book} for more information on asymptotic geometric analysis, isotropic convex bodies and log-concave probability measures. The study of classes of functions from a geometric point of view is a rapidly developing area of research; it is presented in \cite[Chapter~9]{AGA-book-2} where the reader may find a detailed exposition of the main ideas and several important functional inequalities that have been established.

\section{Affine and dual affine quermassintegrals}\label{section-3}

\pa The following inequality about sections of convex bodies was proved by Busemann and Straus \cite{Busemann-Straus-1960}, and
independently by Grinberg \cite{Grinberg-1991}. If $K$ is a convex body in ${\mathbb R}^n$ then, for any $1\ls k\ls n-1$,
\begin{equation}\label{eq:grin-1}\int_{G_{n,k}}|K\cap E|^nd\nu_{n,k}(E)\ls \frac{\omega_k^n}{\omega_n^k}\,
|K|^k.\end{equation}
Following Grinberg's argument one can check that this inequality is still true if we consider
any bounded Borel set $K$ in ${\mathbb R}^n$. This more general form appears in \cite[Section 7]{Gardner-2007}.
Grinberg also observed that the integral in the left-hand side of \eqref{eq:grin-1}
is invariant under volume preserving linear transformations of $K$.

Let $1\ls k\ls n-1$. For every convex body $K$, or more generally for any bounded Borel set in $\mathbb{R}^n$,
the $k$-th dual affine quermassintegral of $K$ is defined by
\begin{equation*}
\Psi_k(K):= \left(\int_{G_{n,k}} |K\cap E^{\perp}|^n \,d\nu_{n,k}(E) \right)^{\frac{1}{kn}}.
\end{equation*}
Grinberg's inequality shows that if $K$ is a bounded Borel set in ${\mathbb R}^n$ and $B_K$ is the centered
Euclidean ball with $|B_K|=|K|$ then
\begin{equation}\label{eq:dual-Phi-upper-bound}
\Psi_k(K)\ls \Psi_k(B_K) =\left(\frac{\omega_{n-k}^n}{\omega_n^{n-k}}\right)^{\frac{1}{kn}}|K|^{\frac{n-k}{kn}}\ls \sqrt{e}|K|^{\frac{n-k}{kn}}
\end{equation}
(the last inequality follows from the fact that $1<\omega_{n-k}/\omega_n^{\frac{n-k}{n}}<e^{k/2}$; see \cite[Lemma~2.1]{Koldobsky-Lifshits-2000} for a proof).
Assuming that $K$ is a centered convex body in ${\mathbb R}^n$ there are two lower bounds on $\Psi_k(K)$, proved in \cite{DP-Phi}:
one has that
\begin{equation}\label{eq:dual-Phi-lower-bounds}
\Psi_k(K) \gr \frac{c}{\psi_{n,k}}|K|^{\frac{n-k}{kn}},
\end{equation}
where $\psi_{n,k}:=\min\left\{L_n,\left(\frac{n}{k}\log(en/k)\right)^{\frac{1}{2}}\right\}$. Since it is now known that $L_n$ is bounded, we conclude
that $\psi_{n,k}$ can be absorbed into an absolute constant $c_1>0$ in what follows. We consider the following natural
generalization of $\Psi_k$ for a non-negative bounded and integrable function $f:{\mathbb R}^n\to [0,+\infty)$:
$$\Psi_k(f):=\left(\int_{G_{n,k}} \|f\,\big|_{E^{\perp}}\|_1^n \,d\nu_{n,k}(E) \right)^{\frac{1}{kn}}.$$
Our aim is to give upper and lower bounds for $\Psi_k(f)$ in terms of $\|f\|_1$.
A functional version of \eqref{eq:grin-1} is established in \cite{Dann-Paouris-Pivovarov-2015}: if $1\ls m \ls n-1$
and $f$ is a non-negative, bounded and integrable function on $\mathbb{R}^n$ then
\begin{equation}\label{eq:intro-grin-6}
\int_{G_{n,m}} \frac{\|f|_H\|_1^n}{\|f|_H\|_\infty^{n-m}} \,d\nu_{n,m}(H) \ls \frac{\omega_m^n}{\omega_n^m}\|f\|_1^m.
\end{equation}
It follows that
\begin{align*}\Psi_k(f)&=\left(\int_{G_{n,n-k}} \|f\,\big|_H\|_1^n \,d\nu_{n,n-k}(H) \right)^{\frac{1}{kn}}\ls \|f\|_{\infty}^{1/n}\left(\int_{G_{n,n-k}} \frac{\|f|_H\|_1^n}{\|f|_H\|_\infty^{k}} \,d\nu_{n,n-k}(H) \right)^{\frac{1}{kn}}\\
&\ls\|f\|_{\infty}^{1/n}\left(\frac{\omega_{n-k}^n}{\omega_n^{n-k}}\right)^{\frac{1}{kn}}\|f\|_1^{\frac{n-k}{kn}}
=\sqrt{e}\|f\|_{\infty}^{1/n}\|f\|_1^{\frac{n-k}{kn}}.
\end{align*}
In the next theorem we give an independent proof of the estimate $\Psi_k(f)\ls \sqrt{e}\|f\|_{\infty}^{1/n}\|f\|_1^{\frac{n-k}{kn}}$.
Moreover, we show that if $f$ is log-concave then a reverse inequality is also true.

\begin{theorem}\label{th:dual-functional}Let $f:{\mathbb R}^n\to [0,+\infty)$ be a bounded integrable function. Then,
$$\Psi_k(f)\ls \sqrt{e}\|f\|_{\infty}^{1/n}\|f\|_1^{\frac{n-k}{kn}}.$$
If $f$ is also assumed to be log-concave and $f(0)>0$ then
$$\Psi_k(f)\gr cf(0)^{1/n}\|f\|_1^{\frac{n-k}{kn}}$$
where $c>0$ is an absolute constant.
\end{theorem}

\begin{proof}Applying \eqref{eq:quotient-2} we write
\begin{align*}
\Psi_k(f) &=\left(\int_{G_{n,n-k}} (f(0)|K_{n-k}(f)\cap H|)^n \,d\nu_{n,n-k}(H) \right)^{\frac{1}{kn}}= f(0)^{\frac{1}{k}}\Psi_k(K_{n-k}(f))\\
&\ls \sqrt{e}f(0)^{\frac{1}{k}}\,|K_{n-k}(f)|^{\frac{n-k}{kn}}
=\sqrt{e}f(0)^{\frac{1}{k}}\left(\frac{|K_{n-k}(f)|}{|K_n(f)|}\right)^{\frac{n-k}{kn}}|K_n(f)|^{\frac{n-k}{kn}}\\
&\ls \sqrt{e}f(0)^{\frac{1}{k}}\left(\frac{\|f\|_{\infty}}{f(0)}\right)^{\frac{n-k}{k}\left(\frac{1}{n-k}-\frac{1}{n}\right)}
\left(\frac{\|f\|_1}{f(0)}\right)^{\frac{n-k}{kn}}=\sqrt{e}\|f\|_{\infty}^{\frac{1}{n}}\|f\|_1^{\frac{n-k}{kn}},
\end{align*}
using in the last steps \eqref{eq:dual-Phi-upper-bound} for the star body $K_{n-k}(f)$, Lemma~\ref{lem:any-g} and
\eqref{eq:quotient-1}.

For the lower bound, we apply \eqref{eq:quotient-2} to write
\begin{align*}
\Psi_k(f) &=\left(\int_{G_{n,n-k}} (f(0)|K_{n-k}(f)\cap H|)^n \,d\nu_{n,n-k}(H) \right)^{\frac{1}{kn}}= f(0)^{\frac{1}{k}}\Psi_k(K_{n-k}(f))\\
&\gr cf(0)^{\frac{1}{k}}\,|K_{n-k}(f)|^{\frac{n-k}{kn}}
=c\left(\frac{|K_{n-k}(f)|}{|K_n(f)|}\right)^{\frac{n-k}{kn}}|K_n(f)|^{\frac{n-k}{kn}}\\
&\gr cf(0)^{\frac{1}{k}}\frac{[(n-k)!]^{\frac{1}{k}}}{[n!]^{\frac{n-k}{kn}}}
\left(\frac{\|f\|_1}{f(0)}\right)^{\frac{n-k}{kn}}=c\frac{[(n-k)!]^{\frac{1}{k}}}{[n!]^{\frac{n-k}{kn}}}
f(0)^{\frac{1}{n}}\|f\|_1^{\frac{n-k}{kn}}.
\end{align*}
using in the last steps \eqref{eq:dual-Phi-lower-bounds} for the convex body $K_{n-k}(f)$
(together with the observation that $\psi_{n,k}$ is bounded by an absolute positive constant $c_1$), Lemma~\ref{lem:log-concave-f}
and \eqref{eq:quotient-1}. From Stirling's formula one may check that there exists a constant $c_2>0$ such that, for all $n\gr 2$ and $1\ls k\ls n-1$,
\begin{equation}\label{eq:constant}\frac{[(n-k)!]^{\frac{1}{k}}}{[n!]^{\frac{n-k}{kn}}}
=\left(\frac{\Gamma(n-k+1)^{\frac{1}{n-k}}}{\Gamma(n+1)^{\frac{1}{n}}}\right)^{\frac{n-k}{k}}\gr c_2.
\end{equation}
Therefore,
$$\Psi_k(f)\gr c^{\prime}f(0)^{\frac{1}{n}}\|f\|_1^{\frac{n-k}{kn}},$$
with $c^{\prime}=c_2c$.
\end{proof}

\begin{note*}For any $s\in (-\infty,0)$, using \eqref{eq:s<0} we can modify the proof of Theorem~\ref{th:dual-functional} and extend it to
the densities of $s$-concave measures.
\end{note*}

\begin{theorem}\label{th:dual-functional-s<0}Let $s\in (-\infty,0)$ and let $f:{\mathbb R}^n\to [0,\infty)$ be
a non-negative integrable function which is the density of an $s$-concave measure $\mu$. Then,
$$\frac{c}{\delta_{n,k,s}^{\frac{n-k}{k}}}f(0)^{\frac{1}{n}}\|f\|_1^{\frac{n-k}{kn}}\ls \Psi_k(f) \ls \sqrt{e}\|f\|_{\infty}^{1/n}\|f\|_1^{\frac{n-k}{kn}}.$$
where
$$\delta_{n,k,s}=\frac{(nB(n,-1/s))^{\frac{1}{n}}}{((n-k)B(n-k,k-1/s))^{\frac{1}{n-k}}}$$
and $c>0$ is an absolute constant.
\end{theorem}

\begin{proof}[Sketch of the proof]Note that the upper bound holds for any bounded integrable function $f$. For the lower bound we note
that
$$K_n(f)\subseteq\delta_{n,k,s}K_{n-k}(f)$$
by \eqref{eq:s<0} and then, as in the proof of Theorem~\ref{th:dual-functional}, we write
\begin{align*}
\Psi_k(f) &=\left(\int_{G_{n,n-k}} (f(0)|K_{n-k}(f)\cap H|)^n \,d\nu_{n,n-k}(H) \right)^{\frac{1}{kn}}= f(0)^{\frac{1}{k}}\Psi_k(K_{n-k}(f))\\
&\gr cf(0)^{\frac{1}{k}}\,|K_{n-k}(f)|^{\frac{n-k}{kn}}
=cf(0)^{\frac{1}{k}}\left(\frac{|K_{n-k}(f)|}{|K_n(f)|}\right)^{\frac{n-k}{kn}}|K_n(f)|^{\frac{n-k}{kn}}\\
&\gr cf(0)^{\frac{1}{k}}\frac{1}{\delta_{n,k,s}^{\frac{n-k}{k}}}
\left(\frac{\|f\|_1}{f(0)}\right)^{\frac{n-k}{kn}}\gr \frac{c}{\delta_{n,k,s}^{\frac{n-k}{k}}}f(0)^{\frac{1}{n}}\|f\|_1^{\frac{n-k}{kn}},
\end{align*}
as claimed. \end{proof}

Let $1\ls k\ls n-1$. For every convex body $K$, or more generally for any bounded Borel set in $\mathbb{R}^n$, the $k$-th affine quermassintegral of $K$ is defined by
$$\Phi_k(K) := \left( \int_{G_{n,k}} |P_E(K)|^{-n} \,d\nu_{n,k}(E) \right)^{-\frac{1}{kn}}.$$
Lutwak \cite{Lutwak-1988b} conjectured that, among all convex bodies of volume $1$, the affine quermassintegrals are minimized in the case of the Euclidean ball $D_n$ of volume $1$ and maximized in the case of the regular simplex $S_n$ of volume $1$ in ${\mathbb R}^n$:
\begin{equation}\label{eq:Lutwak-conjecture}
 \Phi_k(D_n) \ls \Phi_k(K) \ls \Phi_k(S_n),
\end{equation}
for every convex body $K$ of volume $1$ in ${\mathbb R}^n$ and every
$1\ls k \ls n-1$. Note that \eqref{eq:Lutwak-conjecture} for $k=1$ is equivalent to
Blaschke-Santal\'{o} inequality and Mahler's conjecture, and that for $k=n-1$ it is equivalent to
Zhang's inequality \cite{Zhang-1991}. It is known that for every convex body $K$ in $\mathbb{R}^n$,
\begin{equation}\label{eq:Phi-known.bounds}
c_1 \sqrt{n/k}\,|K|^{\frac{1}{n}}\ls \Phi_k(K) \ls c_2\sqrt{n/k}\,\phi_{n,k}\,|K|^{\frac{1}{n}}
\end{equation}
for some absolute constants $c_1,c_2>0$, where $\phi_{n,k}:=\min\left\{\log n , n/k\sqrt{\log(en/k)}\right\}$.
The bounds on the right-hand side of \eqref{eq:Phi-known.bounds} were proved in \cite{DP-Phi}.
The second bound is better when $k$ is proportional to $n$. The left-hand side inequality was proved in \cite{PP-small-ball}.
The recent work of E.~Milman and Yehudayoff \cite{EMilman-Yehudayoff}
establishes the sharp lower bound $\Phi_k(D_n) \ls \Phi_k(K)$
and verifies this part of Lutwak's conjecture, including a characterization of the equality cases,
for all values of $k=1,\ldots ,n-1$: ellipsoids are the only local minimizers with respect to the Hausdorff metric.

Recall that given a non-negative measurable function $f:{\mathbb R}^n\to [0,\infty)$ and $E\in G_{n,k}$, the orthogonal
projection of $f$ onto $E$ is the function $P_Ef:E\to [0,\infty)$ defined by
$$(P_Ef)(z)=\sup\{f(y+z):y\in E^{\perp}\}.$$
It is not hard to check that
$$R_t(P_Ef)=P_E(R_t(f))$$
for every $t>0$, where, for a bounded non-negative integrable function $g$ on ${\mathbb R}^n$, we set
$R_t(g)=\{x:g(x)\gr t\}$. Note also that if $f=\mathds{1}_K$ where $K$ is a compact subset of ${\mathbb R}^n$, then $P_Ef=\mathds{1}_{P_E(K)}$.

Assume that for every $t>0$ the set $R_t(f)$ is compact. A definition of the $k$-th
affine quermassintegral of $f$ was proposed in \cite{Dann-Paouris-Pivovarov-flag}: for every $1\ls k\ls n-1$ let
$$\Phi^{\prime}_k(f):=\int_0^{\infty}\Phi_k(R_t(f)))\,dt
=\int_0^{\infty}\left(\int_{G_{n,k}}|P_E(R_t(f))|^{-n}d\nu_{n,k}(E)\right)^{-\frac{1}{kn}}\,dt.$$
One can check that if $f_r(x):=f(x/r)$ for $r>0$ and $(f\circ T)(x)=f(T^{-1}(x))$ for $T\in GL(n)$ then
$$R_t(f\circ T)=T(R_t(f))\quad\hbox{and}\quad R_t(f_r)=rR_t(f)$$
for any $t>0$. From the $1$-homogeneity and the affine invariance of the usual affine quermassintegrals, we check that
\begin{equation}\label{eq:r-T}\Phi^{\prime}_k(f_r)=r\Phi^{\prime}_k(f)\quad\hbox{and}\quad \Phi^{\prime}_k(f\circ T)=\Phi^{\prime}_k(f)\end{equation}
for any $1\ls k\ls n-1$ and any affine volume preserving transformation $T$ and $r>0$. It was proved in
\cite[Theorem~5.5]{Dann-Paouris-Pivovarov-flag}
that $\Phi_k^{\prime}(f)\gr \Phi^{\prime}_k(f^{\ast})$, where $f^{\ast}$ is the symmetric
decreasing rearrangement of $f$.

We consider the following variant of $\Phi^{\prime}_k(f)$:
$$\Phi_k(f):=\left(\int_0^{\infty}(\Phi_k(R_t(f)))^n\,dt\right)^{1/n}.$$
Note that \eqref{eq:r-T} continues to hold with $\Phi_k^{\prime}(f)$ replaced by $\Phi_k(f)$
and that
\begin{align*}
\Phi_k(\mathds{1}_K) &=\left(\int_0^{\infty}(\Phi_k(R_t(\mathds{1}_K)))^n\,dt\right)^{1/n}\\
&=\left(\int_0^{\infty}\left(\int_{G_{n,k}}|P_E(R_t(\mathds{1}_K))|^{-n}d\nu_{n,k}(E)\right)^{-\frac{1}{k}}\,dt\right)^{1/n}\\
&=\left(\int_0^1\left(\int_{G_{n,k}}|P_E(K)|^{-n}d\nu_{n,k}(E)\right)^{-\frac{1}{k}}\,dt\right)^{1/n}=\Phi_k(K)
\end{align*} for every convex body $K$ in ${\mathbb R}^n$. Moreover,
if $f$ is a geometric log-concave integrable function then $R_t(f)=\varnothing$ for all $t>1$, and hence
$$\Phi_k^{\prime}(f)=\int_0^1\Phi_k(R_t(f)))\,dt\ls \left(\int_0^1(\Phi_k(R_t(f)))^n\,dt\right)^{1/n}=\Phi_k(f)$$
by H\"{o}lder's inequality. We shall prove a two-sided estimate for $\Phi_k(f)$.

\begin{theorem}\label{th:functional}Let $f:{\mathbb R}^n\to [0,+\infty)$ be a bounded integrable function
such that $R_t(f)$ is a compact convex set for all $t>0$. Then,
$$c_1\sqrt{n/k}\,\|f\|_1^{\frac{1}{n}}\ls \Phi_k(f)\ls c_2\sqrt{n/k}\,\phi_{n,k}\,\|f\|_1^{\frac{1}{n}}$$
for every $1\ls k\ls n-1$, where $c_1,c_2>0$ are absolute constants.
\end{theorem}

\begin{proof}Note that
$$\Phi_k(f)=\left(\int_0^{\infty}\left(\int_{G_{n,k}}|P_E(R_t(f))|^{-n}d\nu_{n,k}(E)\right)^{-1/k}\,dt\right)^{1/n}.$$
From \eqref{eq:Phi-known.bounds} we know that
$$c_1 \sqrt{n/k}|R_t(f)|\ls \left(\int_{G_{n,k}}|P_E(R_t(f))|^{-n}d\nu_{n,k}(E)\right)^{-1/k}
\ls c_2\sqrt{n/k}\phi_{n,k}|R_t(f)|$$
for any function $f$ such that $R_t(f)$ is a compact convex set for all $t>0$. Since
$$\left(\int_0^{\infty}|R_t(f)|\,dt\right)^{1/n}=\left(\int_0^{\infty}|\{f\gr t\}|\,dt\right)^{1/n}=\|f\|_1^{\frac{1}{n}}$$
we immediately obtain the result.\end{proof}

\noindent {\it Note.} For any log-concave integrable function $f:{\mathbb R}^n\to [0,+\infty)$ it is clear that $R_t(f)=\{f\gr t\}$ is
compact and convex for all $t>0$ (we easily see that $R_t(f)$ is bounded using the fact that there exist constants $A,B>0$ such that $f(x)\ls Ae^{-B|x|}$ for all $x\in {\mathbb R}^n$; see \cite[Lemma~2.2.1]{BGVV-book}). Therefore, Theorem~\ref{th:functional} holds true for all $f\in {\mathcal F}({\mathbb R}^n)$.

\section{Sections versus projections}\label{section-4}

In this section we present a proof of the following more general version of Theorem~\ref{th:vitali}.

\begin{theorem}\label{th:mixed}Let $f:{\mathbb R}^n\to [0,+\infty)$ be a geometric log-concave integrable function
and $g:{\mathbb R}^n\to [0,+\infty)$ be a bounded integrable function with $\|g\|_{\infty}=1$.
Assume that for some $1\ls k\ls n-1$ we have that
$$\|P_Ef\|_1\ls \|g\big|_E\|_1\qquad\hbox{for all}\;\;E\in G_{n,n-k}.$$
Then,
$$\|f\|_1\ls \frac{n!}{[(n-k)!]^{\frac{n}{n-k}}}\|g\|_1.$$
\end{theorem}

The proof makes use of the notion of mixed integrals of functions, introduced in \cite{VMilman-Rotem-2013}
(see also \cite{Bobkov-Colesanti-Fragala-2014}). First, recall the definition of mixed volumes.
By a classical theorem of Minkowski, if $K_1,\ldots ,K_m$ are non-empty, compact convex
subsets of ${\mathbb R}^n$, then the volume of $\lambda_1K_1+\cdots +\lambda_mK_m$ is a homogeneous polynomial of degree $n$ in
$\lambda_i>0$. One can write
\begin{equation*}|\lambda_1K_1+\cdots +\lambda_mK_m|=\sum_{1\ls i_1,\ldots ,i_n\ls m}
V(K_{i_1},\ldots ,K_{i_n})\lambda_{i_1}\cdots \lambda_{i_n},\end{equation*}
where the coefficients $V(K_{i_1},\ldots ,K_{i_n})$ are invariant under permutations of their arguments. The coefficient $V(K_{i_1},\ldots ,K_{i_n})$ is the mixed volume of $K_{i_1},\ldots ,K_{i_n}$. In particular, if $K$ and $T$ are two convex bodies in ${\mathbb R}^n$
then the function $|K+\lambda T|$ is a polynomial in $\lambda\in [0,\infty )$:
\begin{equation*}|K+\lambda T|=\sum_{k=0}^n \binom{n}{k} V_{n-k}(K,T)\;\lambda^k,\end{equation*}
where $V_{n-k}(K,T)= V((K,n-k),(T,k))$ is the $k$-th mixed volume of $K$ and $T$ (we use  the notation $(T,k)$ for the
$k$-tuple $(T,\ldots ,T)$). If $T=B_2^n$ then we set $W_k(K):=V_{n-k}(K,B_2^n)=V((K, n-k), (B_2^n, k))$; this is the
$k$-th quermassintegral of $K$. Kubota's integral formula expresses the quermassintegral $W_k(K)$ as an average of the volumes of
$(n-k)$-dimensional projections of $K$:
\begin{equation}\label{eq:kubota}W_k(K)=\frac{\omega_n}{\omega_{n-k}}\int_{G_{n,n-k}}
|P_E(K)|d\nu_{n,n-k}(E).\end{equation}
Aleksandrov's inequalities (see \cite{Schneider-book}) imply that if we set
\begin{equation}\label{eq:aleksandrov-1}Q_k(K)=\left
(\frac{W_{n-k}(K)}{\omega_n}\right )^{\frac{1}{k}}=\left
(\frac{1}{\omega_k}\int_{G_{n,k}}|P_E(K)|\,d\nu_{n,k}(E)\right )^{\frac{1}{k}},\end{equation}
then $k\mapsto Q_k(K)$ is decreasing.

In \cite{VMilman-Rotem-2013} the definition of mixed volumes was extended to the setting of functions. Given a bounded
integrable function $f:{\mathbb R}^n\to [0,\infty)$ recall that $R_t(f)=\{f\gr t\}$ for every $t>0$.
If $f_1,\ldots ,f_n:{\mathbb R}^n\to [0,\infty)$ are bounded integrable functions such that
$R_t(f_j)$ is bounded and convex for all $1\ls j\ls n$ and all $t>0$, their mixed integral is defined by
$$V(f_1,f_2,\ldots ,f_n):=\int_0^{\infty}V(R_t(f_1),R_t(f_2),\ldots ,R_t(f_n))\,dt,$$
where we agree that $V(R_t(f_1),R_t(f_2),\ldots ,R_t(f_n))=0$ if $R_t(f_j)=\varnothing$ for some $1\ls j\ls n$. This functional extension reduces 
to the classical mixed volume when $f_j=\mathds{1}_{K_j}$. Note that
$$V(f,\ldots ,f)=\int_0^{\infty}|\{f\gr t\}|\,dt=\|f\|_1.$$
The $k$-th quermassintegral of $f$ is defined as
$$W_k(f)=V((f,n-k),(\mathds{1}_{B_2^n},k)).$$
V.~Milman and Rotem proved in \cite{VMilman-Rotem-2013} that if $u(x)=e^{-|x|}$ then for every
log-concave function $f$ with $f(0)=\|f\|_{\infty}=1$ the function
$k\mapsto \left(W_k(f)/W_k(u)\right)^{\frac{1}{n-k}}$, $0\ls k\ls n-1$
is increasing. In particular,
\begin{equation}\label{eq:aleksandrov-2}\left(\frac{\|f\|_1}{\|u\|_1}\right)^{\frac{n-k}{n}}\ls\frac{W_k(f)}{W_k(u)}\end{equation}
for all $1\ls k\ls n-1$.

It is not hard to compute $\|u\|_1$ and, more generally, $W_k(u)$. Note that, by the definition,
\begin{align*}
W_k(u) &=V((u,n-k),(\mathds{1}_{B_2^n},k)) =\int_0^{\infty}V(R_t(u),\ldots ,R_t(u),R_t(\mathds{1}_{B_2^n}),\ldots ,R_t(\mathds{1}_{B_2^n}))\,dt\\
&=\int_0^1V(R_t(u),\ldots ,R_t(u),B_2^n,\ldots ,B_2^n)\,dt
=\int_0^{\infty}e^{-s}V(R_{e^{-s}}(u),\ldots ,R_{e^{-s}}(u),B_2^n,\ldots ,B_2^n)\,ds\\
&=\int_0^{\infty}e^{-s}V(sB_2^n,\ldots ,sB_2^n,B_2^n,\ldots ,B_2^n)\,ds
=\omega_n\int_0^{\infty}s^{n-k}e^{-s}ds=(n-k)!\omega_n,
\end{align*}
and similarly
$$\|u\|_1=V(u,\ldots ,u)=n!\omega_n.$$
Therefore, \eqref{eq:aleksandrov-2} takes the form
\begin{equation}\label{eq:aleksandrov-3}\|f\|_1^{\frac{n-k}{n}}
\ls\frac{(n!)^{\frac{n-k}{n}}}{(n-k)!\omega_n^{\frac{k}{n}}}W_k(f).\end{equation}

\begin{proof}[Proof of Theorem~$\ref{th:mixed}$]Using the assumption and Kubota's formula we write
\begin{align*}W_k(f) &= \int_0^{\infty}V((R_t(f),n-k),(R_t(B_2^n),k))\,dt\\
&=\frac{\omega_n}{\omega_{n-k}}\int_0^{\infty}\int_{G_{n,n-k}}
|P_E(R_t(f))|d\nu_{n,n-k}(E)\,dt\\
&=\frac{\omega_n}{\omega_{n-k}}\int_{G_{n,n-k}}\left(\int_0^{\infty}
|R_t(P_Ef)|\,dt\right)\,d\nu_{n,n-k}(E)\\
&=\frac{\omega_n}{\omega_{n-k}}\int_{G_{n,n-k}}\left(\int_0^{\infty}
|\{x:P_Ef(x)\gr t\}|\,dt\right)\,d\nu_{n,n-k}(E)\\
&=\frac{\omega_n}{\omega_{n-k}}\int_{G_{n,n-k}}\|P_E(f)\|_1\,d\nu_{n,n-k}(E)\\
&\ls \frac{\omega_n}{\omega_{n-k}}\int_{G_{n,n-k}}\|g\big|_E\|_1\,d\nu_{n,n-k}(E).
\end{align*}
From the proof of Theorem~\ref{th:dual-functional} we see that
\begin{align}\label{eq:quermass}
\int_{G_{n,n-k}}\|g\big|_E\|_1\,d\nu_{n,n-k}(E) &\ls \left(\int_{G_{n,n-k}}\|g\big|_E\|_1^n\,d\nu_{n,n-k}(E)\right)^{1/n}\\
\nonumber &\ls \frac{\omega_{n-k}}{\omega_n^{\frac{n-k}{n}}}\|g\|_{\infty}^{k/n}\|g\|_1^{\frac{n-k}{n}}
=\frac{\omega_{n-k}}{\omega_n^{\frac{n-k}{n}}}\|g\|_1^{\frac{n-k}{n}}.
\end{align}
Combining the above with \eqref{eq:aleksandrov-3} we get
\begin{equation*}
\|f\|_1^{\frac{n-k}{n}} \ls \frac{(n!)^{\frac{n-k}{n}}}{(n-k)!\omega_n^{\frac{k}{n}}}W_k(f)\ls \frac{(n!)^{\frac{n-k}{n}}}{(n-k)!\omega_n^{\frac{k}{n}}}\frac{\omega_n}{\omega_{n-k}}
\frac{\omega_{n-k}}{\omega_n^{\frac{n-k}{n}}}\|g\|_1^{\frac{n-k}{n}}=\frac{(n!)^{\frac{n-k}{n}}}{(n-k)!}\|g\|_1^{\frac{n-k}{n}}.
\end{equation*}
This proves the theorem.
\end{proof}

\section{Sections of log-concave functions}\label{section-5}

The classical Busemann-Petty problem \cite{Busemann-Petty-1956} asks if for any pair of origin-symmetric convex bodies $K,T$ in $\R^n$ that satisfy
\begin{equation}\label{eq:1.1}
|K\cap\xi^\perp | \ls |T\cap\xi^\perp|\end{equation}
for all $\xi\in S^{n-1}$ it follows that $|K| \ls |T|$. It is known that the answer is affirmative if $n\ls 4$, and it is negative if $n\gr 5$ (the history of the problem is presented in the books of Koldobsky \cite{Koldobsky-book} and Gardner \cite{Gardner-book}).
The isomorphic Busemann-Petty problem, posed in \cite{VMilman-Pajor-1989}, asks whether the inequalities \eqref{eq:1.1} imply $\left|K\right| \ls C\left|T\right|,$ where $C$ is an absolute constant.

The lower dimensional Busemann-Petty problem is the following question: Let $1\ls k\ls n-1$ and let $\beta_{n,k}$ be the smallest constant $\beta >0$ with the following property: For every pair of centered convex bodies $K$ and $T$ in ${\mathbb R}^n$ that satisfy
\begin{equation}\label{eq:1.6}|K\cap F|\ls |T\cap F|\end{equation}
for all $F\in G_{n,n-k}$, one has
\begin{equation}\label{eq:1.7}|K|^{\frac{n-k}{n}}\ls \beta^k\,|T|^{\frac{n-k}{n}}.\end{equation}
Then, one may ask if it is true that there exists an absolute constant $C>0$ such that $\beta_{n,k}\ls C$ for all $n$ and $k$.
Until recently, the best known bounds for the constants $\beta_{n,k}$ were the following: For every $1\ls k\ls n-1$ we have
$\beta_{n,k}\ls c_1L_n$ where $c_1>0$ is an absolute constant and we also have the bound $\beta_{n,k}\ls c_2\sqrt{n/k}\,(\log (en/k))^{\frac{3}{2}}$ where $c_2>0$ is an absolute constant (which is stronger for codimensions $k$ that are
proportional to $n$). In fact, the last upper bound follows from the inequality $\beta_{n,k}\ls d_{{\rm {ovr}}}(K,{\mathcal BP}_k^n)$,
where $d_{{\rm {ovr}}}(K,{\mathcal BP}_k^n)$ is the outer volume ratio distance from $K$ to the class ${\mathcal BP}_k^n$ of $k$-intersection bodies in ${\mathbb R}^n$, and the estimate $d_{{\rm {ovr}}}(K,{\mathcal BP}_k^n)\ls c\sqrt{n/k}\ln^{3/2}(en/k)$ obtained in \cite{Koldobsky-Paouris-Zymonopoulou-2011}. Since it is now known that the isotropic constant $L_n$ is bounded by an absolute constant, it follows that $\beta_{n,k}$ is bounded by an absolute constant $C>0$ as well.

An extension of the Busemann-Petty problem to arbitrary measures in place of volume was considered in \cite{Zvavitch-2005}.
Let $K,T$ be origin-symmetric convex bodies in $\R^n,$ and let $f$ be a locally integrable non-negative function on $\R^n.$
Suppose that for every $\xi\in S^{n-1}$
\begin{equation}\label{eq:1.12}\int_{K\cap \xi^\bot} f(x) dx \ls \int_{T\cap \xi^\bot} f(x) dx,\end{equation}
where integration is with respect to Lebesgue measure on $\xi^\bot.$
Does it necessarily follow that
\begin{equation}\label{eq:main-problem}\int_K f(x) dx \ls s_n\int_T f(x) dx\end{equation}
where the constant $s_n$ does not depend on $f,K,T?$ It was proved in \cite{Zvavitch-2005} that, for any strictly positive function $f$, if one asks for such an inequality with constant $1$ then the solution is the same as in the case of volume (where $f\equiv 1$): affirmative if $n\ls 4$ and negative if $n\gr 5$. However, it was proved in \cite{Koldobsky-Zvavitch-2015} that the answer to the isomorphic question is affirmative, namely $s_n\ls \sqrt{n}.$ It is not known whether the $\sqrt{n}$ estimate is optimal.

A lower dimensional version of inequality \eqref{eq:main-problem} was also proved in \cite{Koldobsky-AIM-2015}. If $K$ is a star body in $\R^n,$ $f$ is a continuous non-negative function on $K$, and $1\ls k\ls n-1$, then
\begin{equation}\label{eq:main-problem-2}
\int_Kf(x)\,dx\ \ls C^k \ (d_{\rm {ovr}}(K,{\mathcal{BP}}_k^n))^k\  |K|^{k/n}\ \max_{E\in G_{n,n-k}} \int_{K\cap E} f(x)\,dx,
\end{equation}
where $C$ is an absolute constant. This inequality is an immediate consequence of the following
more general result that was proved in \cite{Giannopoulos-Koldobsky-Zvavitch-2022}.

\begin{theorem}\label{th:intro-quotient}Let $K$ and $T$ be star bodies in $\R^n,$ let $0<k<n$ be an integer,
and let $f,g$ be non-negative continuous functions on $K$ and $T$, respectively, so that $\|g\|_\infty=g(0)=1.$
Then,
\begin{equation}\label{eq:1.13}
\frac{\int_Kf(x),dx}{\left(\int_T g(x)\,dx\right)^{\frac{n-k}n}|K|^{\frac kn}}  \ls\frac n{n-k} \left(d_{\rm ovr}(K,{\mathcal{BP}}_k^n)\right)^k
\max_{E\in G_{n,n-k}} \frac{\int_{K\cap E} f(x)\,dx}{\int_{T\cap E} g(x)\,dx}.
\end{equation}
\end{theorem}

We shall obtain an analogue of Theorem~\ref{th:intro-quotient}, with improved constants, under the assumption that $f$ is log-concave.

\begin{theorem}\label{th:quotient}Let $f,g:{\mathbb R}^n\to [0,\infty)$ be non-negative integrable functions such that
$f$ is log-concave with $f(0)>0$ and $\|g\|_\infty=g(0)=1$.
If $K$ is a convex body in ${\mathbb R}^n$ and $T$ is a compact subset of ${\mathbb R}^n$ with $0\in K\cap T$, then
\begin{equation}\label{eq:quotient}
\left(\frac{f(0)^{-1}\int_Kf(x)\,dx}{\int_T g(x)\,dx}\right)^{\frac{n-k}{n}}  \ls C^k
\max_{E\in G_{n,n-k}} \frac{f(0)^{-1}\int_{K\cap E} f(x)\,dx}{\int_{T\cap E} g(x)\,dx},
\end{equation}
where $C>0$ is an absolute constant.
\end{theorem}

\begin{proof}The function $f_K=f\cdot\mathds{1}_K$ is log-concave, with $f_K(0)=f(0)$.
From \eqref{eq:quotient-1} and \eqref{eq:quotient-2} we immediately get
\begin{equation}\label{eq:quotient-1-lc}
f(0)|K_n(f_K)|=\int_Kf(x)\,dx
\end{equation}
and, for any $1\ls k\ls n-1$ and $E\in G_{n,n-k}$,
\begin{equation}\label{eq:quotient-2-lc}
f(0)|K_{n-k}(f_K)\cap E|=\int_{K\cap E}f(x)\,dx.
\end{equation}
Similarly, for the function $g_T=g\cdot\mathds{1}_T$, using also the assumption that $g_T(0)=1$, we check that
\begin{equation}\label{eq:quotient-3}|K_n(g_T)| =\int_Tg(x)\,dx\quad\hbox{and}\quad
|K_{n-k}(g_T)\cap E| =\int_{T\cap E}g(x)\,dx\end{equation}
for every $E\in G_{n,n-k}$.
We set $$\tau^{n-k}:=f(0)\quad\hbox{and}\quad\varrho^{n-k}:=\max_{E\in G_{n,n-k}} \frac{\int_{K\cap E} f(x)\,dx}{\int_{T\cap E} g(x)\,dx}.$$
From \eqref{eq:quotient-2-lc} and \eqref{eq:quotient-3} we see that
$$|\tau K_{n-k}(f_{K})\cap E|\ls |\varrho K_{n-k}(g_{T})\cap E|$$
for all $E\in G_{n,n-k}$. Therefore,
$$\tau^{n-k}|K_{n-k}(f_K)|^{\frac{n-k}{n}}=|\tau K_{n-k}(f_K)|^{\frac{n-k}{n}}\ls C^k\,|\varrho K_{n-k}(g_T)|^{\frac{n-k}{n}}=C^k\,\varrho^{n-k}|K_{n-k}(g_T)|^{\frac{n-k}{n}}.$$
This shows that
$$|K_{n-k}(f_K)|\ls \tau^{-n}C^{\frac{kn}{n-k}}\varrho^n|K_{n-k}(g_T)|.$$
We apply Lemma~\ref{lem:log-concave-f} to write
\begin{equation}\label{eq:quotient-4} \frac{[(n-k)!]^{\frac{1}{n-k}}}{[n!]^{\frac{1}{n}}}K_n(f_K)\subseteq
K_{n-k}(f_K)\end{equation}
and Lemma~\ref{lem:any-g} together with the assumption that $g(0)=\|g\|_{\infty}=1$ to write
\begin{equation}\label{eq:quotient-5}K_{n-k}(g_T)\subseteq
\left(\frac{\|g\|_{\infty}}{g(0)}\right)^{\frac{1}{n-k}-\frac{1}{n}}K_{n}(g_T)=K_n(g_T).
\end{equation}
Taking into account \eqref{eq:quotient-1} and \eqref{eq:quotient-3} we get
\begin{align*}
\int_Kf(x)\,dx &=f(0)\,|K_n(f_K)|\ls \frac{n!}{[(n-k)!]^{\frac{n}{n-k}}}\,\tau^{n-k}\,|K_{n-k}(f_K)|\\
&\ls \frac{n!}{[(n-k)!]^{\frac{n}{n-k}}}\, \tau^{n-k}\tau^{-n} C^{\frac{kn}{n-k}}\varrho^{n}|K_{n-k}(g_T)|\\
&\ls \frac{n!}{[(n-k)!]^{\frac{n}{n-k}}}\, \tau^{-k} C^{\frac{kn}{n-k}}\varrho^{n}|K_n(g_T)|\\
&= \frac{n!}{[(n-k)!]^{\frac{n}{n-k}}}\,\tau^{-k} C^{\frac{kn}{n-k}}\varrho^{n}
\int_Tg(x)\,dx\\
&= \frac{n!}{[(n-k)!]^{\frac{n}{n-k}}}\,f(0)^{-\frac{k}{n-k}} C^{\frac{kn}{n-k}}\varrho^{n}
\int_Tg(x)\,dx.
\end{align*}
Using Stirling's formula we see that
$$\frac{n!}{[(n-k)!]^{\frac{n}{n-k}}}\ls {c_1}^{\frac{kn}{n-k}},$$
and combining the above we get
$$\left(\frac{\int_Kf(x)\,dx}{\int_Tg(x)\,dx}\right)^{\frac{n-k}{n}}\ls f(0)^{-\frac{k}{n}}(c_1C)^k\varrho^{n-k}
=f(0)^{-\frac{k}{n}}(c_1C)^k\max_{E\in G_{n,n-k}} \frac{\int_{K\cap E} f(x)\,dx}{\int_{T\cap E} g(x)\,dx}$$
as claimed.\end{proof}

Let $f\in {\mathcal F}_0({\mathbb R}^n)$ (more generally, assume that $f$ is log-concave with $f(0)=1$).
Given a measurable function $g:{\mathbb R}^n\to [0,\infty)$ with $\|g\|_{\infty}=g(0)=1$ and a pair $(K,T)$ where $K$ is
a convex body, $T$ is compact and $0\in K\cap T$, we may apply Theorem~\ref{th:quotient} to get
$$\frac{\int_{K}f(x)\,dx}{\left(\int_{T}g(x)\,dx\right)^{\frac{n-k}{n}}}\ls
C^k\left(\int_{K}f(x)\,dx\right)^{\frac{k}{n}}
\max_{E\in G_{n,n-k}} \frac{\int_{K\cap E} f(x)\,dx}{\int_{T\cap E} g(x)\,dx}.$$
Moreover, if we assume that $f\in {\mathcal F}_0({\mathbb R}^n)$ then
$$\int_{K}f(x)\,dx\ls |K|.$$
The same inequality holds true, by Jensen's inequality, if $f$ is log-concave and centered with $f(0)=1$.
Combining the above we get an analogue of Theorem~\ref{th:quotient} involving the volume of $K$.

\begin{corollary}\label{cor:quotient-1}Let $f\in {\mathcal F}_0({\mathbb R}^n)$ and let $g$ be
a non-negative measurable function such that $\|g\|_\infty=g(0)=1$.
If $K$ is a convex body in ${\mathbb R}^n$ and $T$ is a compact subset of ${\mathbb R}^n$ with $0\in K\cap T$, then
\begin{equation}\label{eq:cor-quotient}
\frac{\int_Kf(x)\,dx}{\left(\int_T g(x)\,dx\right)^{\frac{n-k}{n}}}  \ls C^k
|K|^{\frac{k}{n}}\max_{E\in G_{n,n-k}} \frac{\int_{K\cap E} f(x)\,dx}{\int_{T\cap E} g(x)\,dx},
\end{equation}
where $C>0$ is an absolute constant.
\end{corollary}

Theorem~\ref{th:quotient} has a number of consequences if we make a specific choice of $g$ and/or $T$. First, we
can obtain a comparison theorem for the Radon transform. If, in addition to the conditions of Corollary~\ref{cor:quotient-1},
we assume that
$$\int_{K\cap E}f(x)\,dx\ls \int_{T\cap E} g(x)\,dx$$
for all $E\in G_{n,n-k}$, then we get
$$\int_K f(x)\,dx \ls C^k |K|^{\frac kn}\left(\int_Tg(x)\,dx\right)^{\frac{n-k}n}.$$
This is a sharpening (in the log-concave setting) of a result established in \cite{Koldobsky-Paouris-Zvavitch-2022}.

From Corollary~\ref{cor:quotient-1} we also obtain the next lower estimate for the sup-norm of the Radon transform.

\begin{theorem}\label{th:slicing}Let $f\in {\mathcal F}_0({\mathbb R}^n)$. Then, for every $1\ls k\ls n-1$ we have that
$$\int_Kf(x)\,dx\ls C^k |K|^{\frac kn}\max_{E\in G_{n,n-k}} \int_{K\cap E} f(x)\,dx$$
where $C>0$ is an absolute constant.
\end{theorem}

\begin{proof}The theorem follows directly from Corollary~\ref{cor:quotient-1} if we choose $T=B_2^n$ and $g\equiv 1$. We get
$$\int_Kf(x)\,dx \ls \frac{\omega_n^{\frac{n-k}{n}}}{\omega_{n-k}}C^k |K|^{\frac kn}\max_{E\in G_{n,n-k}} \int_{K\cap E} f(x)\,dx$$
and recall that $\omega_n^{\frac{n-k}{n}}/\omega_{n-k}\ls 1$ (see \cite[Lemma~2.1]{Koldobsky-Lifshits-2000}).\end{proof}

Theorem~\ref{th:slicing} is a sharpening (in the log-concave setting) of Koldobsky's slicing inequality for arbitrary functions from \cite{Koldobsky-AIM-2015}. The terminology comes from the slicing problem of Bourgain asking if there exists an absolute constant $C_1>0$ such that for every $n\gr 1$ and every centered convex body $K$ in ${\mathbb R}^n$ one has
\begin{equation}\label{eq:1.2}|K|^{\frac{n-1}{n}}\ls C_1\,\max_{\xi\in S^{n-1}}\,|K\cap \xi^{\perp }|.\end{equation}
It is well-known that this problem is equivalent to the question if $L_n\ls C$.
Consider the best constant $\alpha_{n,k}>0$ with the following property: For every centered convex body $K$ in ${\mathbb R}^n$ one has
\begin{equation}\label{eq:1.5}|K|\ls \alpha_{n,k}^k\,|K|^{\frac{k}{n}}\max_{E\in G_{n,n-k}}|K\cap E|.\end{equation}
The lower dimensional slicing problem asks if there exists an absolute constant $\alpha >0$ such that $\alpha_{n,k}\ls \alpha$
for all $n$ and $k$. Theorem~\ref{th:slicing} states that a more general inequality than \eqref{eq:1.5} holds
for geometric log-concave functions.

Another application of Corollary~\ref{cor:quotient-1} is a mean value inequality for the Radon transform.
Given $f\in {\mathcal F}_0({\mathbb R}^n)$, if we choose $K=T$ and $g\equiv 1$ we get
$$\frac{1}{|K|}\int_Kf(x)\,dx\ls C^k \max_{E\in G_{n,n-k}} \frac{1}{|K\cap E|}\int_{K\cap E}f(x)\,dx.$$
Finally, choosing $f\equiv 1$ and $g\equiv 1$ in Theorem~\ref{th:quotient} we see that if $K$ is a convex body and $T$ is a compact subset of $\R^n$ with $0\in K\cap T$ then,
$$\left(\frac{|K|}{|T|}\right)^{\frac {n-k}n}\ls C^k \max_{E\in G_{n,n-k}} \frac{|K\cap E|}{|T\cap E|}.$$
for every $1\ls k\ls n-1$.

\begin{note*}For any $s\in (-\infty,0)$, using \eqref{eq:s<0} we can modify the proof of Theorem~\ref{th:quotient} and extend it (together
with all its consequences) to the densities of $s$-concave measures. An extension of Corollary~\ref{cor:quotient-1} in the same
spirit, but with a less direct proof, was obtained in \cite{Wu-2020}.
\end{note*}

\begin{theorem}\label{th:quotient-s<0}Let $s\in (-\infty,0)$ and let $f,g:{\mathbb R}^n\to [0,\infty)$ be non-negative integrable functions such that $f$ is the density of an $s$-concave measure $\mu$, $f(0)>0$ and $\|g\|_\infty=g(0)=1$.
If $K$ is a convex body in ${\mathbb R}^n$ and $T$ is a compact subset of ${\mathbb R}^n$ with $0\in K\cap T$, then
\begin{equation}\label{eq:quotient-s<0}
\left(\frac{f(0)^{-1}\int_Kf(x)\,dx}{\int_T g(x)\,dx}\right)^{\frac{n-k}{n}}  \ls \delta_{n,k,s}^{n-k}C^k
\max_{E\in G_{n,n-k}} \frac{f(0)^{-1}\int_{K\cap E} f(x)\,dx}{\int_{T\cap E} g(x)\,dx},
\end{equation}
where
$$\delta_{n,k,s}=\frac{(nB(n,-1/s))^{\frac{1}{n}}}{((n-k)B(n-k,k-1/s))^{\frac{1}{n-k}}}.$$
\end{theorem}

\begin{proof}[Sketch of the proof]We consider the functions $f_K=f\cdot\mathds{1}_K$ and $g_T=g\cdot\mathds{1}_T$.
Note that $f_K$ is $(-1/\alpha)$-concave, where $\alpha =n-\frac{1}{s}>n$.
From \eqref{eq:s<0} we get
\begin{equation}\label{eq:quotient-4-s<0}  K_n(f_K)\subseteq\delta_{n,k,s}K_{n-k}(f_K),\end{equation}
and Lemma~\ref{lem:any-g} together with the assumption that $g(0)=\|g\|_{\infty}=1$ gives
\begin{equation}\label{eq:quotient-5-s<0}K_{n-k}(g_T)\subseteq
\left(\frac{\|g\|_{\infty}}{g(0)}\right)^{\frac{1}{n-k}-\frac{1}{n}}K_{n}(g_T)=K_n(g_T).
\end{equation}
Let $$\varrho^{n-k}:=\max_{E\in G_{n,n-k}} \frac{\int_{K\cap E} f(x)\,dx}{\int_{T\cap E} g(x)\,dx}.$$
Taking into account \eqref{eq:quotient-1} and \eqref{eq:quotient-3} we get
\begin{align*}
\int_Kf(x)\,dx &=f(0)\,|K_n(f_K)|\ls \delta_{n,k,s}^nf(0)\,|K_{n-k}(f_K)|\ls f(0)^{-\frac{k}{n-k}}\delta_{n,k,s}^nC^{\frac{kn}{n-k}}\varrho^{n}|K_{n-k}(g_T)|\\
&\ls f(0)^{-\frac{k}{n-k}}\delta_{n,k,s}^n\varrho^{n}C^{\frac{kn}{n-k}}|K_n(g_T)|= f(0)^{-\frac{k}{n-k}}\delta_{n,k,s}^nC^{\frac{kn}{n-k}}\varrho^{n}
\int_Tg(x)\,dx.
\end{align*}
It follows that
$$\left(\frac{\int_Kf(x)\,dx}{\int_Tg(x)\,dx}\right)^{\frac{n-k}{n}}\ls f(0)^{-\frac{k}{n}}\delta_{n,k,s}^{\frac{n-k}{n}}C^k\varrho^{n-k}
=f(0)^{-\frac{k}{n}}\delta_{n,k,s}^{n-k}C^k\max_{E\in G_{n,n-k}} \frac{\int_{K\cap E} f(x)\,dx}{\int_{T\cap E} g(x)\,dx}$$
as claimed.\end{proof}

\begin{remark}\rm We close this section by mentioning the next result from \cite{CGL} which provides a different general estimate
for the Busemann-Petty problem in the case where $f=g$ is an even log-concave density: If $K$ is a symmetric convex body in ${\mathbb R}^n$ and $T$ is a compact subset of ${\mathbb R}^n$ such that
\begin{equation}\label{eq:intro-14}\int_{K\cap E}f(x)\,dx\ls \int_{T\cap E}f(x)\,dx\end{equation}
for all $E\in G_{n,n-k}$, then
\begin{equation}\label{eq:intro-15}\int_Kf(x)\,dx\ls \big(ckL_{n-k}\big)^{k}\int_Tf(x)\,dx,\end{equation}
where $c>0$ is an absolute constant.
\end{remark}

\section{Projections of log-concave functions}\label{section-6}

In this section we obtain a functional inequality related to the classical Shephard problem.
The original question is the following: Let $K$ and $T$ be two centrally symmetric convex bodies in
${\mathbb R}^n$ such that $|P_{\xi^{\perp}}(K)|<|P_{\xi^{\perp}}(T)|$ for every $\xi\in S^{n-1}$.
Does it follow that $|K|<|T|$? Although the answer is affirmative if $n=2$ (simply because
the assumptions imply that $K\subset T$), it is negative in higher dimensions as shown, independently,
by Petty who gave an explicit counterexample in ${\mathbb R}^3$, and by Schneider for all $n\gr 3$.
It is natural to ask for the order of growth
(as $n\to\infty$) of the best constant
$S_n$ for which the assumptions of Shephard's problem imply that $|K|\ls S_n|T|$.
Recall that the projection body $\Pi K$ of a convex body $K$ is
the centrally symmetric convex body whose support function is defined by
\begin{equation*}h_{\Pi K} (\xi)=|P_{\xi^{\perp} }(K)|,\qquad\xi\in S^{n-1}.\end{equation*}
An argument which is based on the identity
$$V(K,\ldots ,K,\Pi T)=V(T,\ldots ,T,\Pi K)$$
as well as on John's theorem and the fact that ellipsoids are projection bodies, shows that if $K$ and $T$ are two centrally symmetric convex bodies in ${\mathbb R}^n$ such that $|P_{\xi^{\perp}}(K)|<|P_{\xi^{\perp}}(T)|$ for every $\xi\in S^{n-1}$ then
\begin{equation}\label{eq:Pi}|K|^{{\frac{n-1}{n}}}\ls \sqrt{n}|T|^{{\frac{n-1}{n}}}.\end{equation}
This shows that $S_n\ls c_1\sqrt{n}$ for some absolute constant $c_1>0$. In fact, a result of K.~Ball \cite{Ball-1991}
shows that, conversely, there exists an absolute constant $c_2>0$ such that $S_n \gr c_2\sqrt{n}$. 	
We refer to \cite[Section~4.6.2]{AGA-book-2} for a concise, but more detailed, exposition of all the above and references.

In view of the above, one may consider the next lower dimensional Shephard problem. Let $1\ls k\ls n-1$
and let $S_{n,k}$ be the smallest constant $S >0$ with the following property: For
every pair of convex bodies $K$ and $T$ in ${\mathbb R}^n$ that satisfy
$|P_E(K)|\ls |P_E(T)|$ for all $E\in G_{n,n-k}$, one has that $|K|^{\frac{n-k}{n}}\ls S^k\,|T|^{\frac{n-k}{n}}$.
Goodey and Zhang \cite{GZ} proved that $S_{n,k}>1$ if $n-k>1$. General estimates for $S_{n,k}$ were
obtained in \cite{Giannopoulos-Koldobsky-2017}: If $K$ and $T$ are two convex bodies in ${\mathbb R}^n$ such that
$|P_E(K)|\ls |P_E(T)|$ for every $E\in G_{n,n-k}$, then
\begin{equation*}|K|^{\frac{1}{n}}\ls c_1\tilde{S}_{n,k}\,|T|^{\frac{1}{n}},\end{equation*}
where $c_1>0$ is an absolute constant, and
$$\tilde{S}_{n,k}=\min\left\{\sqrt{\tfrac{n}{n-k}}\ln\left (\tfrac{en}{n-k}\right),\ln n\right\}.$$
It follows that $S_{n,k}\ls (c_1\tilde{S}_{n,k})^{\frac{n-k}{k}}$, and in particular that $S_{n,k}\ls C^{\frac{n-k}{k}}$
if $\frac{k}{n-k}$ is bounded.

The next lemma provides a third estimate for $S_{n,k}$, which is reasonably good if $k$ is small.

\begin{lemma}\label{lem:third-estimate}For every $1\ls k\ls n-1$ we have that $S_{n,k}\ls \sqrt{n}$.
\end{lemma}

\begin{proof}We use inductively the following claim. For any $1\ls s\ls n-1$ let ${\alpha}_{n-s}$ be the best
positive constant with the property that $|P_E(K)|\ls\alpha_{n-s} |P_E(T)|=
|P_E(\alpha_{n-s}^{\frac{1}{n-s}}T)|$ for all $E\in G_{n,n-s}$, where the equality follows from homogeneity of volume. Then,
\begin{equation}\label{eq:claim-proj}{\alpha}_{n-s+1}\ls \alpha_{n-s}^{\frac{n-s+1}{n-s}}(n-s+1)^{\frac{n-s+1}{2(n-s)}}.\end{equation}
To see this, consider any $F\in G_{n,n-s+1}$. Note that if $E\in G_{n,n-s}$ and $E\subset F$ then
$P_E(P_F(K))=P_E(K)$ and $P_E(P_F(T))=P_E(T)$. Since the assumption is true for all $1$-codimensional subspaces
of $F$, from \eqref{eq:Pi} we see that
$$|P_F(K)|^{\frac{n-s}{n-s+1}}\ls (n-s+1)^{\frac{1}{2}}|P_F(\alpha_{n-s}^{\frac{1}{n-s}}T)|^{\frac{n-s}{n-s+1}}=(n-s+1)^{\frac{1}{2}}
\alpha_{n-s}|P_F(T)|^{\frac{n-s}{n-s+1}},$$
therefore $|P_F(K)|\ls\alpha_{n-s}^{\frac{n-s+1}{n-s}}(n-s+1)^{\frac{n-s+1}{2(n-s)}}|P_F(T)|$ for all $F\in G_{n,n-s+1}$.
Assume that ${\alpha}_{n-k}=1$, i.e. $|P_E(K)|\ls |P_E(T)|$ for all $E\in G_{n,n-k}$. From \eqref{eq:claim-proj} we see that
$\alpha_{n-k+s}\ls\prod_{j=1}^{s}(n-k+j)^{\frac{n-k+s}{2(n-k+j-1)}}$, and in particular,
$${\alpha}_n\ls\prod_{j=1}^k(n-k+j)^{\frac{n}{2(n-k+j-1)}}\ls n^{\frac{n}{2}\sum_{j=1}^k\frac{1}{n-k+j-1}}\ls n^{\frac{nk}{2(n-k)}}.$$
This means that $|K|\ls n^{\frac{nk}{2(n-k)}}|T|$, therefore
$$|K|^{\frac{n-k}{n}}\ls n^{\frac{k}{2}}|T|^{\frac{n-k}{n}},$$
which shows that $S_{n,k}\ls\sqrt{n}$.
\end{proof}

\begin{note*}It is not hard to check that if $k\ll\frac{n}{\log n}$ then $\sqrt{n}\ll \tilde{S}_{n,k}^{\frac{n-k}{k}}$,
and hence, in order to estimate $S_{n,k}$ in this range it is preferable to use the upper bound of Lemma~\ref{lem:third-estimate}.
\end{note*}

Our goal is to obtain an analogue of the above estimates for $S_{n,k}$ in the setting of geometric log-concave
integrable functions. Our substitute for the assumption that all $k$-codimensional projections of $K$ gave
smaller volume than the corresponding projections of $T$ is the following: we consider $f,g\in {\mathcal F}_0({\mathbb R}^n)$
such that, for some $1\ls k\ls n-1$ and for all $E\in G_{n,n-k}$ and $t>0$ we have that
\begin{equation}\label{eq:assumption}|R_t(P_Ef)|\ls |R_t(P_Eg)|.\end{equation}
Recall that \eqref{eq:assumption} can be equivalently written as $|P_E(R_t(f))|\ls |P_E(R_t(g))|$. Moreover,
if $f=\mathds{1}_K$ and $g=\mathds{1}_T$ are the indicator functions of two convex bodies in ${\mathbb R}^n$ then
for every $0\ls t\ls 1$ we have that $R_t(P_E\mathds{1}_K)=P_E(K)$ and $R_t(P_E\mathds{1}_T)=P_E(T)$,
and hence \eqref{eq:assumption} is equivalent to the assumption $|P_E(K)|\ls |P_E(T)|$
in Shephard's problem.

\begin{theorem}\label{th:shephard}Let $f,g\in {\mathcal F}_0({\mathbb R}^n)$ and $1\ls k\ls n-1$. Assume that $|R_t(P_Ef)|\ls |R_t(P_Eg)|$ for all $E\in G_{n,n-k}$ and $0\ls t\ls 1$. Then,
$$\|f\|_1^{\frac{n-k}{n}}\ls S_{n,k}^k\|g\|_1^{\frac{n-k}{n}}.$$
\end{theorem}

\begin{proof}The assumption $|R_t(P_Ef)|\ls |R_t(P_Eg)|$ for all $E\in G_{n,n-k}$ implies that $|R_t(f)|\ls S_{n,k}^{\frac{kn}{n-k}}|R_t(g)|$ for all $0\ls t\ls 1$ by the definition of $S_{n,k}$. Then,
$$\|f\|_1=\int_0^1|R_t(f)|\,dt\ls S_{n,k}^{\frac{kn}{n-k}}\int_0^1|R_t(g)|\,dt=S_{n,k}^{\frac{kn}{n-k}}\|g\|_1$$
and the result follows.\end{proof}

\bigskip

\noindent {\bf Acknowledgement.} We thank the referee for helpful comments and constructive suggestions that improved the presentation of the paper. 
The author acknowledges support by the Hellenic Foundation for Research and Innovation (H.F.R.I.) in the framework of the 
call ``Basic research Financing (Horizontal support of all Sciences)” under the National Recovery and Resilience Plan ``Greece 2.0” funded by
the European Union –NextGenerationEU (H.F.R.I. Project Number: 15445).

\bigskip


\footnotesize
\bibliographystyle{amsplain}


\bigskip

\medskip

\thanks{\noindent {\bf Keywords:} log-concave functions, sections and projections, affine and dual-affine quermassintegrals.

\smallskip

\thanks{\noindent {\bf 2020 MSC:} Primary 52A20; Secondary 52A40, 52A39, 26B25.}

\bigskip

\bigskip

\noindent \textsc{Natalia \ Tziotziou}: School of Applied Mathematical and Physical Sciences, National Technical University of Athens, Department of Mathematics, Zografou Campus, GR-157 80, Athens, Greece.

\smallskip

\noindent \textit{E-mail:} \texttt{nataliatz99@gmail.com}

\end{document}